\documentclass{birkjour}
\usepackage{amssymb,amsmath,amsopn}
\usepackage{url}

\newtheorem{thm}{Theorem}[section]
 \newtheorem{cor}[thm]{Corollary}
 \newtheorem{lem}[thm]{Lemma}
 
 \newtheorem{statement}[thm]{Statement}
 \theoremstyle{definition}
 \newtheorem{defn}[thm]{Definition}
 \theoremstyle{remark}
 \newtheorem{rem}[thm]{Remark}
 
 \numberwithin{equation}{section}

\begin{document}

\title[]{Hyperbolic plane geometry revisited}
\author[\'A. G.Horv\'ath]{\'Akos G.Horv\'ath}
\date{25 April, 2014}

\address{\'A. G.Horv\'ath, Dept. of Geometry, Budapest University of Technology,
Egry J\'ozsef u. 1., Budapest, Hungary, 1111}
\email{ghorvath@math.bme.hu}


\begin{abstract}
Using the method of C. V\"or\"os, we establish results in hyperbolic plane geometry, related to triangles and circles.
We present a model independent construction for Malfatti's problem and several trigonometric formulas for triangles.
\end{abstract}

\maketitle

\section{Introduction}

J. W. Young, the editor of the book \cite{johnson}, wrote in his introduction:
\emph{ There are fashions in mathematics as well as in clothes, -- and in both domains they have a tendency to repeat themselves.}
During the last decade, ``hyperbolic plane geometry'' aroused much interest and was investigated vigorously by a considerable number of mathematicians.

Despite the large number of investigations, the number of hyperbolic trigonometric formulas that can be collected from them is fairly small,
they can be written on a page of size B5. This observation is very surprising if we compare it with the fact that already in 1889, a very extensive and elegant treatise of spherical trigonometry was written by John Casey \cite{casey 1}.
For this, the reason, probably, is that the discussion of a problem in hyperbolic geometry is less pleasant than in spherical one.

On the other hand, in the 19th century, excellent mathematician -- Cyrill V\"or\"os\footnote{Cyrill V\"or\"os (1868 --1948), piarist, teacher} in Hungary
made a big step to solve this problem. He introduced a method for the measurement of distances and angles in the case that the considered points or lines, respectively, are not real. Unfortunately, since he published his works mostly in Hungarian or in Esperanto, his method is not well-known to the mathematical community.

To fill this gap, we use the concept of distance extracted from his work and, translating the standard methods of Euclidean plane geometry into the language of the hyperbolic plane, apply it for various configurations. We give a model independent construction for the famous problem of Malfatti (discussed in \cite{gho 1}) and give some interesting formulas connected with the geometry of hyperbolic triangles. By the notion of distance introduced by V\"or\"os, we obtain results in hyperbolic plane geometry which are not well-known. The length of this paper is very limited, hence some proofs will be omitted here. The interested reader can find these proofs in the unpublished source file \cite{gho 3}.

\subsection{Well-known formulas on hyperbolic trigonometry}

The points $A,B,C$ denote the vertices of a triangle. The lengths of the edges opposite to these vertices are $a,b,c$, respectively. The angles at $A,B,C$ are denoted by $\alpha, \beta, \gamma$, respectively. If the triangle has a right angle, it is always at $C$. The symbol $\delta$ denotes half of the area of the triangle; more precisely, we have $2\delta=\pi-(\alpha+\beta+\gamma)$.
\begin{itemize}
\item {\bf Connections between the trigonometric and hyperbolic trigonometric functions:}
$$
\sinh a=\frac{1}{i}\sin (ia), \quad \cosh a=\cos (ia), \quad \tanh a=\frac{1}{i}\tan (ia).
$$
\item {\bf Law of sines:}
\begin{equation}
\sinh a :\sinh b :\sinh c=\sin \alpha :\sin \beta :\sin \gamma.
\end{equation}
\item {\bf Law of cosines:}
\begin{equation}
\cosh c=\cosh a\cosh b-\sinh a\sinh b \cos \gamma.
\end{equation}
\item {\bf Law of cosines on the angles:}
\begin{equation}
\cos \gamma=-\cos\alpha\cos\beta+\sin\alpha\sin\beta \cosh c.
\end{equation}
\item {\bf The area of the triangle:}
\begin{equation}
T:=2\delta=\pi-(\alpha+\beta+\gamma).
\end{equation}
\begin{equation}
\tan \frac{T}{2}=\left(\tanh \frac{a_1}{2}+\tanh \frac{a_1}{2}\right)\tanh \frac{m_a}{2},
\end{equation}
where $m_a$ is the height of the triangle corresponding to $A$ and $a_1,a_2$ are the signed lengths of the segments into which the foot point of the height divides the side $BC$.
\item {\bf Heron's formula:}
\begin{equation}
\tan \frac{T}{4}=\sqrt{\tanh \frac{s}{2}\tanh \frac{s-a}{2}\tanh \frac{s-b}{2}\tanh \frac{s-c}{2}}.
\end{equation}
\item{\bf Formulas on Lambert's quadrangle:} The vertices of the quadrangle are $A,B,C,D$ and the lengths of the edges are $AB=a,BC=b,CD=c$ and $DA=d$, respectively. The only angle which is not a right angle is $ BCD\measuredangle=\varphi$. Then, for the sides, we have:
$$
\tanh b=\tanh d\cosh a, \quad \tanh c=\tanh a\cosh d,
$$
and
$$
 \sinh b=\sinh d\cosh c, \quad \sinh c=\sinh a\cosh b .
$$
Moreover, for the angles, we have:
$$
\cos \varphi=\tanh b\tanh c=\sinh a\sinh d, \quad \sin \varphi=\frac{\cosh d}{\cosh b}=\frac{\cosh a}{\cosh c},
$$
and
$$
\tan \varphi =\frac{1}{\tanh a\sinh b}=\frac{1}{\tanh d\sinh c}.
$$
\end{itemize}

\section{The distance of points and on the lengths of segments}

First we extract the concepts of the distance of real points following the method of the book of Cyrill V\"or\"os (\cite{voros}). We extend the plane with two types of points, one of the type of the points at infinity and the other one the type of ideal points. In a projective model these are the boundary and external points of a model with respect to the embedding real projective plane. Two parallel lines determine a point at infinity, and two ultraparallel lines an ideal point which is the pole of their common transversal. Now the concept of the line can be extended; a line is real if it has real points (in this case it also has two points at infinity and the other points on it are ideal points being the poles of the real lines orthogonal to the mentioned one). The extended real line is a closed compact set with finite length. We also distinguish the line at infinity which contains precisely one point at infinity and the so-called ideal line which contains only ideal points. By definition the common lengths of these lines are $\pi ki$, where $k$ is a constant of the hyperbolic plane and $i$ is the imaginary unit. In this paper we assume that $k=1$. Two points on a line determine two segments $AB$ and $BA$.
The sum of the lengths of these segments is $AB+BA=\pi i$. We define the length of a segment as an element of the linearly ordered set $\bar{\mathbb{C}}:=\overline{\mathbb{R}}+ \mathbb{R}\cdot i$. Here $\overline{\mathbb{R}}=\mathbb{R}\cup\{\pm \infty\}$ is the linearly ordered set of real numbers extracted with two new numbers with the "real infinity" $\infty $ and its additive inverse $-\infty$. The infinities  can be considered as new "numbers" having the properties that either "there is no real number greater than or equal to $\infty$" or "there is no real number less than or equal to $-\infty$". We also introduce the following operational rules: $\infty+\infty=\infty$, $-\infty+(-\infty)=-\infty$, $\infty+(-\infty)=0$ and  $\pm\infty+a=\pm\infty$ for real $a$. It is obvious that $\overline{\mathbb{R}}$ is not a group, the rule of associativity holds only for such expressions which contain at most two new objects. In fact, $0=\infty+(-\infty)=(\infty+\infty)+(-\infty)=\infty+(\infty+(-\infty))=\infty$ is a contradiction. We also require that the equality $\pm\infty+b i=\pm\infty +0 i$ holds for every real number $b$, and for brevity we introduce the respective notations $\infty:=\infty +0 i$ and $-\infty:=-\infty +0 i$. We extract the usual definition of hyperbolic function based on the complex exponential function by the following formulas:
$$
\cosh (\pm \infty):=\infty, \sinh (\pm \infty):=\pm \infty, \mbox{ and } \tanh(\pm \infty):=\pm 1.
$$
We also assume that $\infty \cdot \infty=(-\infty)\cdot(-\infty)=\infty$, $\infty\cdot(-\infty)=-\infty$ and $\alpha \cdot (\pm \infty)=\pm \infty$.

Assuming that the trigonometric formulas of hyperbolic triangles are also valid with ideal vertices the definition of the mentioned lengths of the complementary segments of a line are given. For instance, consider a triangle with two real vertices ($B$ and $C$) and an ideal one ($A$), respectively. The lengths of the segments between $C$ and $A$ are $b$ and $b'$, the lengths of the segments between $B$ and $A$ are $c$ and $c'$ and the lengths of that segment between $C$ and $B$ which contains only real points is $a$, respectively. Let the right angle be at the vertex $C$ and denote by $\beta$ the other real angle at $B$. (See in Fig. 1.)

\begin{figure}[htbp]
\centerline{\includegraphics[height=4cm]{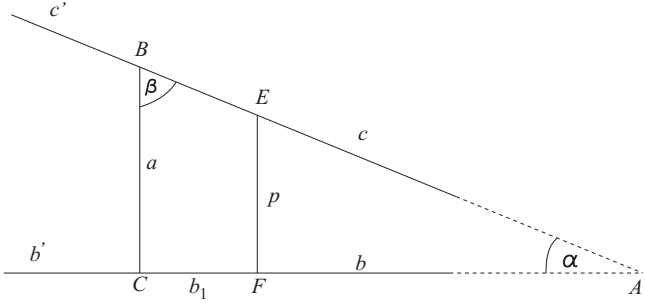}}
\caption{Length of the segments between a real and an ideal point}
\end{figure}

With respect to this triangle we have $\tanh b=\sinh a\cdot \tan \beta$, and since $A$ is an ideal point, the parallel angle corresponding to the distance $\overline{BC}=a$ less than or equal to $\beta $. Hence $\tan \beta > 1/\sinh a$ implying that $\tanh b >1$. Hence $b$ is a complex number. If the polar of $A$ is $EF$, then it is the common perpendicular of the lines $AC$ and $AB$. The quadrangle $CFEB$ has three right angles. Denote by $b_1$ the length of that segment $\overline{CF}$ which contains real points only. Then we get
$
\tan \beta =\frac{1}{\tanh b_1 \sinh a},
$
meaning that
$
\sinh a \tan \beta=\frac{1}{\tanh b_1}=\tanh b.
$
Similarly we have that $\tanh b'=\sinh a\cdot \tan (\pi-\beta)=- \sinh a\cdot \tan \beta$ implying that $|\tanh b'|>1$, hence $b'$ is complex. Now we have that
$
\tanh b'=-\frac{1}{\tanh b_1}.
$
Using the formulas between the trigonometric and hyperbolic trigonometric functions we get that
$
\frac{1}{i}\tan ib =\frac{i}{\tan ib_1},
$
implying that
$
\tan ib =-\tan\left(\frac{\pi}{2}-ib_1\right),
$
so
$
b=-\frac{2n-1}{2}\pi i+b_1.
$
Analogously we get also that
$
b'=-\frac{2m+1}{2}\pi i-b_1.
$
Here $n$ and $m$ are arbitrary integers. On the other hand, if $b_1=0$ then $AC=CA$, and so $b=b'$ meaning that $2n-1=2m+1$. For the half length of the complete line we can choose an odd multiplier of the number $\pi i/2$. The most simple choosing is when we assume that $n=0$ and $m=-1$. Thus the lengths of the segments $AC$ and $CA$ can be defined as
$b=b_1+\frac{\pi}{2}$ and $b'=-b_1+\frac{\pi}{2}$, respectively.

We now define all of the possible lengths of a segment on the basis of the type of the line that contains them.

\subsection{The points $A$ and $B$ are on a real line.}

We can distinguish six subcases. The definitions of the respective cases can be found in Table 1. We abbreviate the words real, infinite and ideal by symbols ${\mathcal R}$, ${\mathcal I}n$ and ${\mathcal I}d$, respectively. $d$ means a real (positive) distance of the corresponding usual real elements which are a real point or the real polar line of an ideal point, respectively. Every box in the table contains two numbers which are the lengths of the two segments determined by the two points. For example, the distance of a real and an ideal point is a complex number. Its real part is the distance of the real point to the polar of the ideal point with a sign. This sign is positive in the case when the polar line intersects the segment between the real and ideal points, and is negative otherwise. The imaginary part of the length is $(\pi/2)i$, implying that the sum of the lengths of two complementary segments of this projective line has total length $\pi i$.
Consider now a point at infinity. This point can also be considered as the limit of real points or limit of ideal points of this line. By definition the distance from a point at infinity of a real line to any other real or infinite point of this line is $\pm \infty$ according to that it contains or not ideal points. If, for instance, $A$ is an infinite point and $B$ is a real one, then the segment $AB$ contains only real points has length $\infty$. It is clear that with respect to the segments on a real line the length-function is continuous.

\begin{table}
\begin{tabular}{cc|c|c|c|}
  \cline{3-5}
  & & \multicolumn{3}{|c|}{$B$} \\ \cline{2-5}
  &\multicolumn{1}{|c|}{ } &  ${\mathcal R}$  & ${\mathcal I}n$  & ${\mathcal I}d$  \\ \cline{1-5}
 \multicolumn{1}{|c|}{}& \multicolumn{1}{|c|}{${\mathcal R}$} & \begin{tabular}{c}  $AB=d$ \\ $BA=-d+\pi i$  \end{tabular}   & \begin{tabular}{c} $AB=\infty $ \\$BA=-\infty$ \end{tabular} & \begin{tabular}{c} $AB=d+\frac{\pi}{2} i$ \\  $BA=-d+\frac{\pi}{2} i$ \end{tabular} \\ \cline{2-5}
 \multicolumn{1}{|c|}{$A$}& \multicolumn{1}{|c|}{${\mathcal I}n$} &   & \begin{tabular}{c} $AB=\infty $ \\ $BA=-\infty$ \end{tabular}  & \begin{tabular}{c} $AB=\infty$ \\$BA=-\infty $ \end{tabular} \\  \cline{2-5}
 \multicolumn{1}{|c|}{}& \multicolumn{1}{|c|}{${\mathcal I}d$} & &  & \begin{tabular}{c}  $AB=d+\pi i$ \\ $BA=-d$ \end{tabular} \\
  \hline
\end{tabular}
\vspace{0,2cm}
\caption{Distances on the real line.}
\end{table}

\subsection{The points $A$ and $B$ are on a line at infinity.}

We can check that the length of a segment for which either $A$ or $B$ is an infinite point is indeterminable. To see this, let the real point $C$ be a vertex of a right-angled triangle whose other vertices $A$ and $B$ are on a line at infinity with infinite point $B$. Then we get that $\cosh c=\cosh a\cdot \cosh b$ for the corresponding sides of this triangle. But from the result of the previous subsection
$$
\cosh a=\cosh \infty=\infty \mbox{ and } \cosh b=\cosh \left(0+\frac{\pi}{2}i\right)=\cos \left(-\frac{\pi}{2}\right)=0,
$$
showing that their product is undeterminable. On the other hand, if we consider the polar of the ideal point $A$ we get a real line through $B$. The length of a segment connecting the (ideal) point $A$ and one of the points of its polar is $(\pi/2)i$. This means that we can define the length of a segment between $A$ and $B$ also as this common value. Now if we also want to preserve the additivity property of the lengths of segments on a line at infinity, then we must give the pair of values $0, \pi i$ for the lengths of segment with ideal ends. Table 2 collects these definitions.

\begin{table}
\begin{tabular}{cc|c|c|}
  \cline{3-4}
 & & \multicolumn{2}{|c|}{$B$} \\ \cline{2-4}
      & \multicolumn{1}{|c|}{} & ${\mathcal I}n$  & ${\mathcal I}d$  \\ \cline{1-4}
  \multicolumn{1}{|c|}{$A$} & \multicolumn{1}{|c|}{${\mathcal I}n$}  & \begin{tabular}{c} $AB=0 $ \\ $BA=\pi i$ \end{tabular}  & \begin{tabular}{c} $AB=\frac{\pi}{2}i$ \\ $BA=\frac{\pi}{2}i $ \end{tabular} \\  \cline{2-4}
    \multicolumn{1}{|c|}{} & \multicolumn{1}{|c|}{${\mathcal I}d$} & & \begin{tabular}{c}  $AB=0$ \\ $BA=\pi i$ \end{tabular} \\
  \hline
\end{tabular}
\vspace{0.2cm}
\caption{Distances on the line at infinity.}
\end{table}

\subsection{The points $A$ and $B$ are on an ideal line.}

\begin{figure}[htbp]
\centerline{\includegraphics[height=4cm]{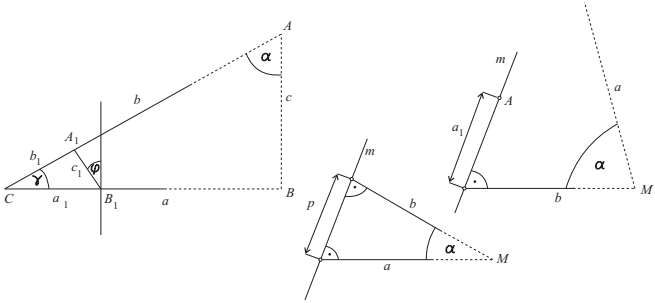}}
\caption{The cases of the ideal segment and angles}
\end{figure}

This situation contains only one case: $A$, $B$ and $AB$ are ideal elements, respectively. We need first the measure of the angle of two real ultraparallel lines. (See $\alpha$ in Fig.1). Then clearly $\cos \alpha=\cosh a \cdot \sin \beta>1$, and so $\alpha $ is imaginary. From Lambert's quadrangle $BCEF$ we get
$$
\cosh a \sin \beta =\cosh p,
$$
thus $\cosh p=\cos \alpha$ and so $\alpha =2n\pi \pm p i $. Now an elementary analysis of the figure shows that the continuity property requires the choice $n=0$. If we also assume that we choose the negative sign, then the measure is $\alpha=-p i=p/i$, where $p$ is the length of that segment of the common perpendicular whose points are real.

Consider now an ideal line and its two ideal points $A$ and $B$, respectively. The polars of these points intersect each other in a real point $B_1$. Consider a further real point $C$ of the line $BB_1$ and denote by $A_1$ the intersection point of the polar of $A$ and the real line $AC$ (see Fig. 2).

Observe that $A_1B_1$ is perpendicular to $AC$; thus we have $\tanh b_1 =\tanh a_1 \cdot \cos \gamma$. On the other hand, $a=\pm a_1 +(\pi i)/2$ and $b=\pm b_1 +(\pi i)/2$ implying that $\tanh b =\tanh a \cdot \cos \gamma$. Hence the angle between the real line $CB$ and the ideal line $AB$ can be considered to $\pi /2$, too. Now from the triangle $ABC$ we get that
$$
\cosh c =\frac{\cosh b}{\cosh a}=\frac{\pm i\sinh b_1}{\pm i\sinh a_1}=\frac{\sinh b_1}{\sinh a_1}=\sin \left(\frac{\pi}{2}-\varphi\right)=\cos \varphi,
$$
where $\varphi$ is the angle of the two polars. From this we get $c=2n\pi\pm \varphi/i=2n\pi \mp \varphi i$. We choose $n=0$ since at this time $\varphi=0$ implies $c=0$ and the positive sign because the length of the line is $\pi i$.

\emph{The length of an ideal segment on an ideal line is the angle of their polars multiplied by the imaginary unit $i$.}

\subsection{Angles of lines}

\begin{table}[htbp]
\begin{tabular}{cc|c|c|c|}
  \cline{3-5} & & \multicolumn{3}{|c|}{$a$} \\ \cline{3-5}
  & & ${\mathcal R}$ & ${\mathcal I}n$ & ${\mathcal I}d$ \\ \hline
  \multicolumn{1}{|c|}{}& \multicolumn{1}{|c|}{${\mathcal R}$} & \begin{tabular}{|c|c|c|} \multicolumn{3}{|c|}{$M$} \\ \hline ${\mathcal R}$ & ${\mathcal I}n$ & ${\mathcal I}d$ \\ \hline
                                    \begin{tabular}{c} $\varphi$ \\
                                    $\pi-\varphi$
                                    \end{tabular} & \begin{tabular}{c} $0$ \\
                                    $\pi$
                                    \end{tabular}  & \begin{tabular}{c} $\frac{p}{i}$ \\
                                    $\pi -\frac{p}{i}$
                                    \end{tabular} \\
                                    \end{tabular} &
                                                    \begin{tabular}{|c|c|} \multicolumn{2}{|c|}{$M$} \\ \hline  ${\mathcal I}n$ & ${\mathcal Id}$ \\ \hline
                                                     \begin{tabular}{c} $\frac{\pi}{2}$ \\
                                                                         $\frac{\pi}{2}$
                                                     \end{tabular}  &
                                                                         \begin{tabular}{c} $\infty$ \\
                                                                                            $-\infty$
                                                                                          \end{tabular} \\
                                                                                           \end{tabular} &
                                                                                                           \begin{tabular}{|c|}  \multicolumn{1}{|c|}{$M$} \\ \hline ${\mathcal I}d$ \\ \hline
                                                                                                           \begin{tabular}{c} $\frac{\pi}{2}+\frac{a_1}{i}$ \\
                                                                                                            $\frac{\pi}{2}-\frac{a_1}{i}$
                                                                                                            \end{tabular} \\
                                                                                                             \end{tabular}\\ \cline{2-5}
   \multicolumn{1}{|c|}{$b$} & \multicolumn{1}{|c|}{${\mathcal I}n$} & &\begin{tabular}{|c|}
                                \multicolumn{1}{|c|}{$M$} \\ \hline ${\mathcal I}d$ \\ \hline
                                    $\infty$ \\
                                    $-\infty$
                                    \end{tabular}
                                                 &\begin{tabular}{|c|} \multicolumn{1}{|c|}{$M$} \\ \hline
                                                   ${\mathcal I}d$ \\ \hline
                                                    $\infty$ \\
                                                     $-\infty$
                                                   \end{tabular}  \\ \cline{2-5}
  \multicolumn{1}{|c|}{} &\multicolumn{1}{|c|}{${\mathcal I}d$} &  &  & \begin{tabular}{|c|} \multicolumn{1}{|c|}{$M$} \\ \hline
                                    ${\mathcal I}d$ \\ \hline
                                    $\frac{p}{i}$ \\
                                    $\pi-\frac{p}{i}$
                                    \end{tabular} \\
  \hline
\end{tabular}
\caption{Angles of lines.}
\end{table}
Similarly as in the previous paragraph we can deduce the angle between arbitrary kinds of lines (see Table 3). In Table 3, $a$ and $b$ are the given lines, $M=a\cap b$ is their intersection point, $m$ is the polar of $M$ and $A$ and $B$ are the poles of $a$ and $b$, respectively. The numbers $p$ and $a_1$ represent real distances, as can be seen on Fig. 2, respectively.
The general connection between the angles and distances is the following:
\emph{ Every distance of a pair of points is the measure of the angle of their polars multiplied by $i$. The domain of the angle can be chosen in such a way, that we are going through the segment by a moving point and look at the domain which is described by the moving polar of this point. }

\subsection{The extracted hyperbolic theorem of sines}

With the above definition of the length of a segment the known formulas of hyperbolic trigonometry can be extracted to the formulas of general objects with real, infinite or ideal vertices. For example, we can prove the hyperbolic theorem of sines for right-angled triangles. It says that $\sinh a=\sinh c\cdot \sin \alpha$.

\begin{figure}[htbp]
\centerline{\includegraphics[scale=0.7]{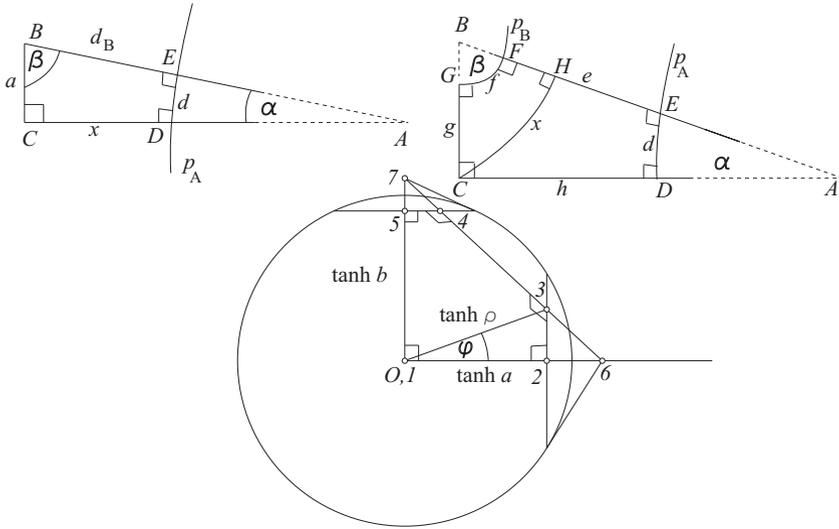}}
\caption{Hyperbolic theorem of sines with non-real vertices}
\end{figure}

We prove {\bf first} those cases when all sides of the triangle lie on real lines, respectively. We assume that the right angle is at $C$ and that it is a real point because of our definition.

\begin{itemize}
\item If $A$ is an infinite point $B$ and $C$ are real ones then $\sinh c\cdot \sin \alpha=\infty \cdot 0$ is indeterminable and we can consider that the equality is true. The relation  $\sinh b\cdot \sin \beta=\infty \cdot \sin \beta =\infty$ is also true by our agreement. If $A,B$ are at infinity then $\alpha =\beta =0$ and the equality holds, too.
\item In the case when $B,C$ are real points and $A$ is an ideal point, let the polar of $A$ be $p_A$. Then by definition
$ \sinh c=\sinh (d_B +(i\pi/2))=\cosh (d_B)\sinh (i\pi/2)=i\cosh (d_B)$ where $d_B$ is the distance of $B$ and $p_a$; $\sin\alpha= \sin (d/i)=i(1/i)\sin(-id)=-i\sinh(d)$ where $d$ is the length of the segment between the lines of the sides $AC$ and $BC$. If $p_A$ intersects $AC$ and $BC$ in the points $D$ and $E$, respectively, then $BCDE$ is a quadrangle with three right angles and with the sides $a$, $x$, $d$ and $d_B$ (see the left figure in Fig. 4). This implies that $ \sinh c \sin\alpha=\cosh (d_B)\sinh(d)=\sinh a $, as we stated.

\item If $C$ is a real point, $A$ is at infinity, and $B$ is an ideal point, then $\alpha  =0$ and the right-hand side $\sinh c\cdot \sin \alpha$ is undeterminable. If we consider $\sinh c\cdot \sin \beta=\infty \sin\beta$ it is infinite by our agreement, and the statement is true, again.

\item Very interesting is the last case when $C$ is a real point, $A$ and $B$ are ideal points, respectively, and the line $AB$ is a real line (see the right-hand side picture in Fig. 3). Then $\sinh a=i\cosh g$, $ \sinh c=\sinh (-e)$ and $\sin\alpha=-i\sinh d$, thus $\sinh c \sin\alpha=i\sinh e \sinh d$ and the theorem holds if and only if in the real pentagon $CDEFG$ with five right angles it holds that $\sinh e \sinh d=\cosh g$. But we have:
\end{itemize}

\begin{statement}[\cite{gho 3}]
Denote by $a,b,c,d,e$ the edge lengths of the successive sides of a pentagon with five right angles on the hyperbolic plane. Then we have:
$$\cosh d=\sinh a\sinh b, \quad \sinh c=\frac{\cosh a}{\sqrt{\sinh^2 a\sinh^2 b-1}}$$
$$\sinh e=\frac{\cosh b}{\sqrt{\sinh^2 a\sinh^2 b-1}}$$.
\end{statement}

{\bf Second} we assume that the hypotenuse $AB$ lies on a non-real line. Now if it is at infinity and at least one vertex is an infinite point then the statement evidently true. Assume that $A$, $B$ and its line are ideal elements, respectively. Then the length $c$ is equal to $(\pi/2)i$, the angle $\alpha$ is equal to $(\pi/2)+d/i$, where $d$ is the distance between $C$, and the polar of $B$ and the length of $a$ is equal to $d+(\pi/2)i$, respectively. The equality $\sinh (\pi/2)i\cdot \sin ((\pi/2)+d/i)=(1/i)\sin (-(\pi/2))\cos (d/i)=-(1/i)\cosh d=i \cosh d=\sinh (d+(\pi/2)i)$ proves the statement in this case, too.

\section{Power, inversion and centres of similitude}

It is not clear who investigated first the concept of inversion with respect to hyperbolic geometry. A synthetic approach can be found in \cite{molnar} using reflections in Bachmann's metric plane. For our purpose it is more convenient to use an analytic approach in which the concepts of centres of similitude and axis of similitude can be defined. We mention that the spherical approach of these concepts can be found in Chapter VI and Chapter VII in \cite{casey 1}.

In the hyperbolic case, using the extracted concepts of lengths of segments, this approach can be reproduced.

\begin{lem}[\cite{gho 3}] The product $\tanh (PA)/2 \cdot \tanh (PB)/2$ is constant if $P$ is a fixed (but arbitrary) point (real, at infinity or ideal), $P,A,B$ are collinear and $A,B$ are on a cycle of the hyperbolic plane (meaning that in the fixed projective model of the real projective plane it has a proper part).
\end{lem}
On the basis of Lemma 3.1. we can define the power of a point with respect to a given cycle.
\begin{defn}
The \emph{power} of a point $P$ with respect to a given cycle is the value
$$
c:=\tanh \frac{1}{2}PA \cdot \tanh \frac{1}{2} PB,
$$
where the points $A$, $B$ are on the cycle, such that the line $AB$ passes through the point $P$. With respect to Lemma 1 this point could be a real, infinite or ideal one. The \emph{ axis of power } of two cycles is the locus of points having the same powers with respect to the cycles.
\end{defn}
The power of a point can be positive, negative or complex. (For example, in the case when $A,B$ are real points we have the following possibilities: it is positive if $P$ is a real point and it is in the exterior of the cycle; it is negative if $P$ is real and it is in the interior of the cycle, it is infinite if $P$ is a point at infinity, or complex if $P$ is an ideal point.)
We can also introduce the concept of similarity center of cycles.
\begin{defn}
The \emph{ centres of similitude} of two cycles with non-overlapping interiors are the common points of their pairs of tangents touching directly or inversely (i.e., they do not separate, or separate the circles), respectively. The first point is the \emph{external center of similitude}, the second one is \emph{the internal center of similitude}.
\end{defn}
For intersecting cycles separating tangent lines do not exist, but the internal center of similitude is defined as on the sphere, but replacing $\sin $ by $\sinh $. More precisely we have
\begin{lem}[\cite{gho 3}]
Two points $S,S'$ which divide the segments $OO'$ and $O'O$, joining the centers of the two cycles in the hyperbolic ratio of the hyperbolic sines of the radii $r,r'$ are the centers of similitude of the cycles. By formula, if $\sinh OS : \sinh SO'=\sinh O'S':\sinh S'O=\sinh r:\sinh r'$ then the points $S,S'$ are the centers of similitude of the given cycles.
\end{lem}
We also have the following
\begin{lem}[\cite{gho 3}]
If the secant through a centre of similitude $S$ meets the cycles in the corresponding points $M, M'$ then
$\tanh \frac{1}{2}SM$ and $\tanh \frac{1}{2} SM'$ are in a given ratio.
\end{lem}
We now discuss the cases for the possible centers of similitude. We have six possibilities.
\begin{description}
\item[i] \textsl{The two cycles are circles.} To get the centers of similitude we have to solve an equation in $x$. Here $d$ means the distance of the centers of the circles, $r\leq R$ denotes the respective radii, and $x$ is the distance of the center of similitude to the center of the circle with radius $r$.$\sinh (d\pm x): \sinh x=\sinh R:\sinh r$ from which we get that $\coth x=\frac{\sinh R\mp\cosh d\sinh r}{\sinh r\sinh d}$ or, equivalently,
    $$
    e^x=\sqrt{\frac{\coth x+1}{\coth x -1}}=\sqrt{\frac{(\sinh R)/(\sinh r)\mp e^{\mp d}}{(\sinh R)/(\sinh r)\mp e^{\pm d}}}.
    $$
    The two centers corresponding to the two cases of possible signs. If we assume that $e^x=\sqrt{\frac{(\sinh R)/(\sinh r)- e^{-d}}{(\sinh R)/(\sinh r)- e^{ d}}}$, then the center is an ideal point, point at infinity or a real point according to the cases $\sinh R/\sinh r <e^d, \quad \sinh R/\sinh r =e^d$, or  $\sinh R/\sinh r>e^d$, respectively. The corresponding center is the external center of similitude. In the other case we have $e^x=\sqrt{\frac{(\sinh R)/(\sinh r)+ e^{d}}{(\sinh R)/(\sinh r)+ e^{- d}}}$, and the corresponding center is always a real point. This is the internal center of similitude.

\item[ii] \textsl{One of the cycles is a circle and the other one is a paracycle.} The line joining their centers (which we call axis of symmetry) is a real line, but the respective ratio is zero or infinite. To determine the centres we have to decide the common tangents and their points of intersections, respectively. The external centre is a real, infinite or ideal point, and the internal centre is a real point.

\item[iii] \textsl{One of the cycles is a circle and the other one is a hypercycle.} The axis of symmetry is a real line such that the ratio of the hyperbolic sines of the radii is complex. The external center is a real, infinite or ideal point, the internal one is always a real point. Each of them can be determined as in the case of two circles.

\item[iv] \textsl{Each of them is a paracycle.} The axis of symmetry is a real line and the internal centre is a real point. The external centre is an ideal point.

\item[v] \textsl{One of them is a paracycle and the other one is a hypercycle.} The axis of symmetry (in the Poincar\'e model, with the hypercycle replaced by the circular line containing it, and the axis containing the two apparent centers) is a real line. The internal centre is a real point. The external centre is a real, infinite or ideal point.

\item[vi] \textsl{Both of them are hypercycles.} The axis of symmetry (in the Poincar\'e model, with the hypercycle replaced by the circular line containing it, and the axis containing the two apparent centers) can be a real line, ideal line or a line at infinity. For the internal centre we have three possibilities as above as well as for the external centre.
\end{description}
We can use the concepts of "axis of similitude", "inverse and homothetic pair of points", "homothetic to and inverse of a curve $\gamma $ with respect to a fixed point $S$ (which "can be real point, a point at infinity, or an ideal point, respectively") as in the case of the sphere. More precisely we have:
\begin{lem}[\cite{gho 3}]
The six centers of similitude of three cycles taken in pairs lie three by three on four lines, called \emph{axes of similitude} of the cycles.
\end{lem}
From Lemma 3.5 it follows immediately that if two pairs of intersection points of a line through $S$ with the cycles are $N,N'$ and $M,M'$ then $\tanh \frac{1}{2}SM \cdot \tanh \frac{1}{2} SN'$ is independent from the choice of the line.
Thus, given a fixed point $S$ (which is the center of the cycle at which we would like to invert) and any curve $\gamma $, on the hyperbolic plane, if on the halfline joining $S$ (the endpoint of the halfline) with any point $M$ of $\gamma $ a point $N'$ is taken, such that
$ \tanh \frac{SM}{2}\cdot \tanh \frac{SN'}{2}$
is constant, the locus of $N'$ is \emph{called the inverse} of $\gamma $. We also use the name \emph{cycle of inversion} for the locus of the points whose squared distance from $S$ is $ \tanh \frac{SM}{2}\cdot\tanh \frac{SN'}{2}$.
Among the projective elements of the pole and its polar either one of them is always real or both of them are at infinity. Thus, in a construction the common point of two lines is well-defined, and in every situation it can be joined with another point; for example, if both of them are ideal points they can be given by their polars (which are constructible real lines) and the required line is the polar of the intersection point of these two real lines. Thus the lengths in the definition of the inverse can be constructed. This implies that the inverse of a point can be constructed on the hyperbolic plane, too.

\begin{rem}Finally we remark that all of the concepts and results of inversion with respect to a sphere of the Euclidean space can be defined also in the hyperbolic space, the "basic sphere" could be a hypersphere, parasphere or sphere, respectively. We can use also the concept of ideal elements and the concept of elements at infinity, if it is necessary. It can be proved (using Poincar\'e's ball-model) that every hyperbolic plane of the hyperbolic space can be inverted to a sphere by such a general inversion. This map sends the cycles of the plane to circles of the sphere.
\end{rem}

\section{Applications}

In this section we give applications, some of them having analogous on the sphere, and others being completely new ones.

\subsection{Steiner's construction on Malfatti's construction problem}

Malfatti (see \cite{malfatti}) raised and solved the following problem:
\emph{construct three circles into a triangle so that each of them touches the two others from outside and, moreover, touches also two sides of the triangle.}

The first nice moment was Steiner's construction. He gave an elegant method (without proof) to construct the given circles. He also extended the problem and his construction to the case of three given circles instead of the sides of a triangle (see in \cite{steiner 1}, \cite{steiner 2}). Cayley referred to this problem in \cite{cayley} as \emph{Steiner's extension of Malfatti's problem}.
We note that Cayley investigated and solved a further generalization in \cite{cayley}, which he also called Steiner's extension of Malfatti's problem. His problem is \emph{to determine three conic sections so that each of them touches the two others, and also touches two of three more given conic sections.}
Since the case of circles on the sphere is a generalization of the case of circles of the plane (as it can be seen easily by stereographic projection), Cayley indirectly proved Steiner's second construction. We also have to mention Hart's nice geometric proof for Steiner's construction which was published in \cite{hart}. (It can be found in various textbooks, e.g. \cite{casey} and also on the web.)

In the paper \cite{gho 1} we presented a possible form of Steiner's construction which meet the original problem in the best way. We note (see the discussion in the proof) that our theorem has a more general form giving all possible solutions of the problem. However, for simplicity we restrict ourself to the most plausible case, when the cycles touch each other from outside. In \cite{gho 1} we used the fact that cycles are represented by circles in the conformal model of Poincar\'e. The Euclidean constructions of circles of this model gives hyperbolic constructions on cycles in the hyperbolic plane. To do these constructions manually we have to use special rulers and calipers to draw the distinct types of cycles. For brevity, we think of a fixed conformal model of the embedding Euclidean plane and preserve the name of the known Euclidean concepts with respect to the corresponding concept of the hyperbolic plane, too. We now interprete this proof without using models. We use Gergonne's construction (see the Euclidean version in \cite{dorrie}, and the hyperbolic one in \cite{gho 1} or \cite{gho 3}) which solves the problem \emph{Construct a circle (cycle) touching three given circles (cycles) of the plane.}

\begin{figure}[htbp]
\centerline{\includegraphics[scale=1]{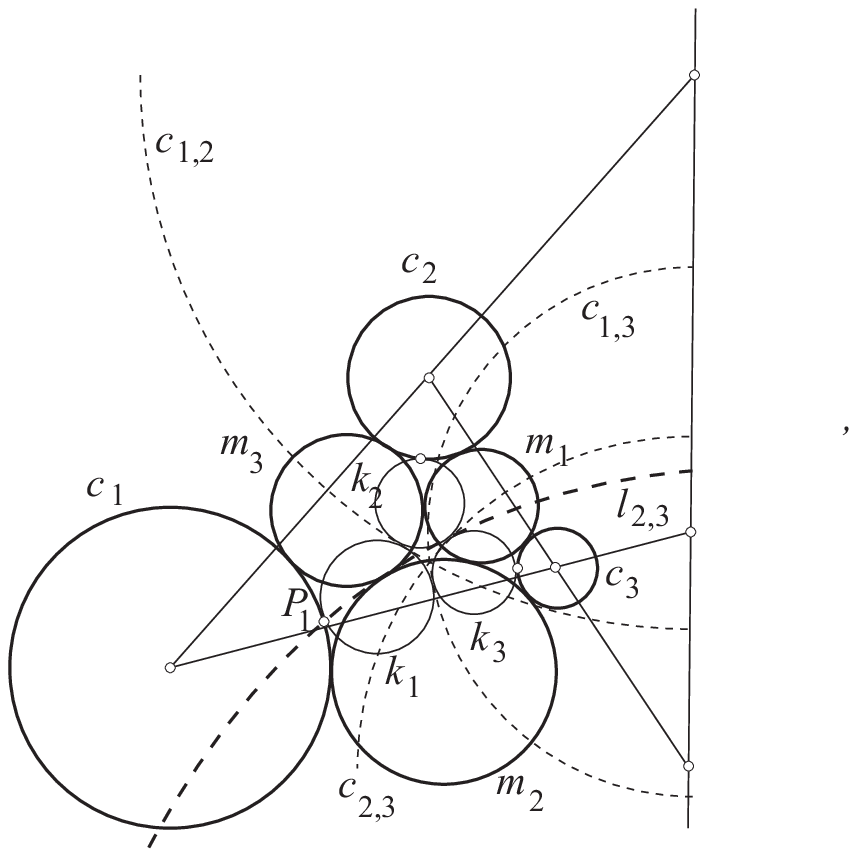}}
\caption{Steiner's construction.}
\end{figure}

\begin{thm}[\cite{gho 1}]
Steiner's construction can be done also in the hyperbolic plane. More precisely, for three given non-overlapping cycles there can be constructed three other cycles, each of them touching the two other ones from outside and also touching two of the three given cycles from outside.
\end{thm}

\begin{proof} Denote by $c_i$ the given cycles. Now the steps of Steiner's construction are the following.
\begin{enumerate}
\item Construct the cycle of inversion $c_{i,j}$, for the given cycles $c_i$ and $c_j$, where the center of inversion is the external centre of similitude of them. (I.e., the center of $c_{i,j}$ is the center of the above inversion, and $c_i,c_j$ are images of each other with respect to inversion at $c_{ij}$. Observe that $c_{ij}$ separates $c_i$ and $c_j$.)
\item Construct the cycle $k_j$ touching two cycles $c_{i,j}, c_{j,k}$ and the given cycle $c_j$, in such a way that $k_j,c_j$ touch from outside, and $k_{ij},c_{ij}$ (or $c_{jk}$) touch in such a way that $k_j$ lies on that side of $c_{ij}$ (or $c_{ik}$) on which side of them $c_j$ lies.
\item Construct the cycle $l_{i,j}$ touching $k_i$ and $k_j$ through the point $P_k=k_k\cap c_k$.
\item Construct Malfatti's cycle $m_j$ as the common touching cycle of the four cycles $l_{i,j}$, $l_{j,k}$, $c_i$, $c_k$.
\end{enumerate}
The first step is the hyperbolic interpretation of the analogous well-known Euclidean construction of circles.

To the second step we follow Gergonne's construction (see in \cite{gho 3}).
The third step is a special case of the second one. (A given cycle is a point now.) Obviously the general construction can be done in this case, too.

The fourth step is again the second one choosing three arbitrary cycles from the four ones, since the quadrangles determined by the cycles have incircles.

Finally we have to prove that this construction gives the Malfatti cycles. As we saw, the Malfatti cycles exist (see in \cite{gho 1} Theorem 1). We also know that in an embedding hyperbolic space the examined plane can be inverted to a sphere. The trigonometry of the sphere is absolute implying that the possibility of a construction which can be checked by trigonometric calculations, is independent of the fact that the embedding space is a hyperbolic space or a Euclidean one. Of course, the Steiner construction is just such a construction, the touching position of circles on the sphere can be checked by spherical trigonometry. So we may assume that the examined sphere is a sphere of the Euclidean space and we can apply Cayley's analytical methods (see in \cite{cayley}) by which he proved that Steiner's construction works on a surface of second order. Hence the above construction produces the required touches.
\end{proof}

\subsection{Applications for triangle centers}

There are many interesting statements on triangle centers. In this section we mention some of them, concentrating only on the centroid, circumcenters and incenters, respectively.

The notation of this subsection follows the previous part of this paper: \emph{the vertices} of the triangle are $A,B,C$, the corresponding angles are $\alpha,\beta,\gamma$ and the lengths of the sides opposite to the vertices are $a,b,c$, respectively. We also use the notion $2s=a+b+c$ for the \emph{perimeter} of the triangle. Let denote $R,r,r_A,r_B,r_C$ the radius of the \emph{circumscribed cycle}, the radius of the \emph{inscribed cycle} (shortly incycle), and the radii of the \emph{escribed cycles} opposite to the vertices $A,B,C$, respectively. We do not assume that the points $A,B,C$ are real and the distances are positive numbers. In most cases the formulas are valid for ideal elements and elements at infinity and when the distances are complex numbers, respectively. The only exception when the operation which needs to the examined formula has no exact mathematical meaning. Before examining hyperbolic triangle centers, we collect some further important formulas on hyperbolic triangles. We can consider them in our extracted manner.
\subsubsection{Staudtian and angular Staudtian of a hyperbolic triangle:}
The concept of \emph{Staudtian of a hyperbolic triangle} somehow similar (but definitely distinct) to the concept of the Euclidean area. In spherical trigonometry the twice of this very important quantity was called by Staudt the sine of the trihedral angle $O-ABC$, and later Neuberg suggested the names (first) ``Staudtian'' and the ``Norm of the sides'', respectively. We prefer in this paper the name ``Staudtian''to honour the great geometer Staudt.
Let
$$
n=n(ABC):=\sqrt{\sinh s\sinh (s-a)\sinh (s-b)\sinh (s-c)}.
$$ 
Then we have
\begin{equation}
\sin \frac{\alpha}{2}\sin \frac{\beta}{2}\sin \frac{\gamma}{2}=\frac{n^2}{\sinh s\sinh a\sinh b\sinh c}.
\end{equation}
This observation leads to the following formulas of the Staudtian:
\begin{equation}
\sin \alpha =\frac{2n}{\sinh b\sinh c}, \quad \sin \beta =\frac{2n}{\sinh a\sinh c}, \quad \sin \gamma =\frac{2n}{\sinh a\sinh b}. \end{equation}
From the first equality of (4.2) we get that
\begin{equation}
n=\frac{1}{2} \sin \alpha \sinh b\sinh c= \frac{1}{2}\sinh h_C\sinh c,
\end{equation}
where $h_C$ is the height of the triangle corresponding to the vertex $C$.
As a consequence of this concept we can give homogeneous coordinates of the points of the plane with respect to a basic triangle as follows:
\begin{defn}
Let $ABC$ be a non-degenerate reference triangle of the hyperbolic plane. If $X$ is an arbitrary point we define its coordinates by the ratio of the Staudtian $X:=\left(n_A(X):n_B(X):n_C(X)\right)$
where $n_A(X)$, $n_B(X)$ and $n_C(X)$ means the Staudtian of the triangle $XBC$, $XCA$ and $XAB$, respectively. This triple of coordinates is the \emph{triangular coordinates} presents the point $X$ with respect to the triangle $ABC$.
\end{defn}
Consider finally the ratio of section $(BX_AC)$ where $X_A$ is the foot of the transversal $AX$ on the line $BC$. If $n(BX_AA)$, $n(CX_AA)$   mean the Staudtian of the triangles $BX_AA$, $CX_AA$, respectively, then, using (4.3), we have
$$
(BX_AC)=\frac{\sinh BX_A}{\sinh X_AC}=\frac{\frac{1}{2}\sinh h_C\sinh BX_A}{\frac{1}{2}\sinh h_C\sinh X_A C}=\frac{n(BX_AA)}{n(CX_AA)}=
$$
$$
=\frac{\frac{1}{2}\sinh c\sinh AX_A\sin(BAX_A)\measuredangle}{\frac{1}{2}\sinh b\sinh AX_A\sin(CAX_A)\measuredangle}=\frac{\sinh c\sinh AX\sin(BAX_A)\measuredangle}{\sinh b\sinh AX\sin(CAX_A)\measuredangle}=\frac{n_C(X)}{n_B(X)},
$$
proving that
\begin{equation}
(BX_AC)=\frac{n_C(X)}{n_B(X)}, (CX_BA)=\frac{n_A(X)}{n_C(X)}, (AX_CB)=\frac{n_B(X)}{n_A(X)}.
\end{equation}

The \emph{angular Staudtian} of the triangle defined by the equality
$$
N=N(ABC):=\sqrt{\sin\delta\sin(\delta+\alpha)\sin(\delta+\beta)\sin(\delta+\gamma)}
$$
is the "dual" of the concept of Staudtian and thus we have similar formulas for it. From the law of cosines for the angles we have
$\cos \gamma=-\cos \alpha\cos \beta + \sin \alpha\sin\beta \cosh c$
and adding to this the addition formula of the cosine function we get that
$$
\sin \alpha\sin\beta(\cosh c-1)=\cos \gamma +\cos (\alpha +\beta)=2\cos\frac{\alpha+\beta+\gamma}{2}\cos\frac{\alpha+\beta-\gamma}{2}.
$$
From this we obtain that
\begin{equation}
\sinh\frac{c}{2}=\sqrt{\frac{\sin{\delta}\sin{(\delta+\gamma})}{\sin \alpha\sin\beta}}.
\end{equation}
Analogously we get that
\begin{equation}
\cosh \frac{c}{2}=\sqrt{\frac{\sin{(\delta+\beta)}\sin{(\delta+\alpha})}{\sin \alpha\sin\beta}}.
\end{equation}
From these equations it follows that
\begin{equation}
\cosh \frac{a}{2}\cosh \frac{b}{2}\cosh \frac{c}{2}=\frac{N^2}{\sin \alpha\sin\beta\sin \gamma\sin\delta}.
\end{equation}
Finally we also have that
\begin{equation} \sinh a =\frac{2N}{\sin \beta\sin \gamma}, \quad \sinh b=\frac{2N}{\sin \alpha\sin \gamma}, \quad \sinh c =\frac{2N}{\sin \alpha\sin \beta}, \end{equation}
and from the first equality of (4.8) we get that
\begin{equation}
N=\frac{1}{2} \sinh a \sin\beta\sin\gamma= \frac{1}{2}\sinh h_C\sin\gamma.
\end{equation}
The connection between the two Staudtians is given by the formula
\begin{equation}
2n^2=N\sinh a\sinh b\sinh c.
\end{equation}
Dividing the first equality of (4.2) by the analogous one in (4.8) we get that $\frac{\sin \alpha}{\sinh a}=\frac{n}{N}\frac{\sin \beta}{\sinh b}\frac{\sin \gamma}{\sinh c}$
implying the equality
\begin{equation}
\frac{N}{n}=\frac{\sin \alpha}{\sinh a}.
\end{equation}

\subsubsection{On the centroid (or median point) of a triangle}
We denote the medians of the triangle by $AM_A,BM_B$ and $CM_C$, respectively. The feet of the medians are $M_A$,$M_B$ and $M_C$. The existence of their common point $M$ follows from the Menelaos theorem (\cite{szasz}). For instance if $AB$, $BC$ and $AC$ are real lines and the points $A,B$ and $C$ are ideal points then we have that $AM_C=M_CB=d=a/2$ implies that $M_C$ is the middle point of the real segment lying on the line $AB$ between the intersection points of the polars of $A$ and $B$ with $AB$, respectively (see Fig. 5).

\begin{figure}[htbp]
\centerline{\includegraphics[scale=0.6]{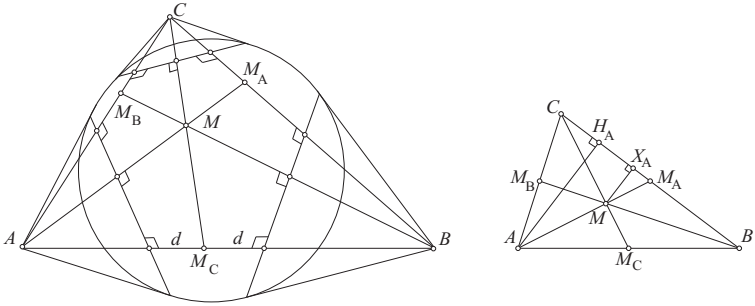}}
\caption{Centroid of a triangle with ideal vertices.}
\end{figure}

The fact that the centroid exists implies new real hyperbolic statements, e.g. \emph{ Consider a real hexagon with six right angles. Then the lines containing the middle points of a side and being perpendicular to the opposite sides of the hexagon are concurrent}.
\begin{thm}[\cite{gho 3}] We have the following formulas connected with the centroid:
\begin{equation}
n_A(M)=n_B(M)=n_C(M),
\end{equation}
\begin{equation}
\frac{\sinh AM}{\sinh MM_A}=2 \cosh \frac{a}{2},
\end{equation}
\begin{equation}
\frac{\sinh AM_A}{\sinh MM_A}=\frac{\sinh BM_B}{\sinh MM_B}=\frac{\sinh CM_C}{\sinh MM_C}=\frac{n}{n_A(M)},
\end{equation}
\begin{equation}
\sinh d'_M=\frac{\sinh d'_A +\sinh d'_B +\sinh d'_C}{\sqrt{1+2(1+\cosh a +\cosh b+\cosh c)}},
\end{equation}
where $d'_A$, $d'_B$, $d'_C$, $d'_M$ mean the signed distances of the points $A,B,C,M$ to a line $y$, respectively.
Finally we have
\begin{equation}
\cosh YM=\frac{\cosh YA +\cosh YB +\cosh YC}{\frac{n}{n_A(M)}}.
\end{equation}
where $Y$ is a point of the plane. (4.15) and (4.16) are called the ``center-of-gravity'' property of $M$ and the ``minimality property'' of $M$, respectively.
\end{thm}
\begin{rem}
Using the first order approximation of the hyperbolic functions by their Taylor polynomial of order 1, we get from this formula the following one: $ d'_M=\frac{d'_A+d'_B+d'_C}{3}$
which associates the centroid with the physical concept of center of gravity and shows that the center of gravity of three equal weights at the vertices of a triangle is at $M$.
\end{rem}

\begin{rem}
The minimality property of $M$ for $Y=M$ says that $\cosh MA +\cosh MB +\cosh MC=\sqrt{1+2(1+\cosh a +\cosh b+\cosh c)}$.
This implies $\cosh YA +\cosh YB +\cosh YC=(\cosh MA +\cosh MB +\cosh MC) \cosh YM$.
From the second-order approximation of $\cosh x $ we get that \hfill \break $3+\frac{1}{2}\left(YA^2+ YB^2+ YC^2\right)=\left(3+\frac{1}{2}\left(MA^2+ MB^2+ MC^2\right)\right)\left(1+\frac{1}{2}YM^2\right)$.
From this (take into consideration only such terms whose order is less than or equal to $2$) we get an Euclidean identity characterizing the centroid: $YA^2+YB^2+YC^2=MA^2+MB^2+MC^2+3YM^2$.
As a further consequence we can see immediately that the value $\cosh YA +\cosh YB +\cosh YC$ is minimal if and only if $Y$ is the centroid.
\end{rem}

\subsubsection{On the center of the circumscribed cycle}

Denote by $O$ the center of the circumscribed cycle of the triangle $ABC$. In the extracted plane $O$ always exists and could be a real point, point at infinity or ideal point, respectively. Since we have two possibilities to choose the segments $AB$, $BC$ and $AC$ on their respective lines, we also have four possibilities to get a circumscribed cycle. One of them corresponds to the segments with real lengths and the others can be gotten if we choose one segment with real length and two segments with complex lengths, respectively. If $A,B,C$ are real points the first cycle could be a circle, a paracycle or a hypercycle, but the other three are always hypercycles, respectively. For example, let $a'=a=BC$ be a real length, and  $b'=-b+\pi i$, $c'=-c +\pi i$ be complex lengths, respectively. Then we denote by $O_A$ the corresponding (ideal) center and by $R_A$ the corresponding (complex) radius. We also note that the latter three hypercycles have geometric meaning. These are those hypercycles whose fundamental lines contain a pair from the midpoints of the edge-segments and contain that vertex of the triangle which is the meeting point of the corresponding edges.
\begin{thm} The following formulas are valid on the circumradii:
\begin{equation} \tanh R=\frac{\sin\delta}{N}, \quad \tanh R_A=\frac{\sin(\delta+\alpha)}{N}, \end{equation}
\begin{equation} \tanh R=\frac{2\sinh\frac{a}{2}\sinh\frac{b}{2}\sinh\frac{c}{2}}{n}, \quad
\tanh R_A=\frac{2\sinh\frac{a}{2}\cosh\frac{b}{2}\cosh\frac{c}{2}}{n}. \\
\end{equation}
\begin{equation}
n_A(0):n_B(O)=\cos (\delta+\alpha)\sinh a:\cos (\delta+\beta)\sinh b
\end{equation}
\end{thm}

\begin{rem}
The first order Taylor polynomial of the hyperbolic functions of distances leads to a correspondence between the hyperbolic Staudtians and the Euclidean area $T$ yealding also further Euclidean formulas. More precisely, we have $n=T$ and $N=\frac{T\sin\alpha}{a}=\frac{Ta}{2Ra}=\frac{T}{2R}$.
Hence we give the following formula: $\sin\alpha\sin\beta\sin\gamma=\frac{2N^2}{n}=\frac{2T^2}{4R^2T}=\frac{T}{2R^2}$
or, equivalently, the known Euclidean dependence of these quantities: $T=2R^2\sin\alpha\sin\beta\sin\gamma$.
\end{rem}

\begin{rem}
Use the minimality property of $M$ for the point $Y=O$. Then we have $\sqrt{1+2(1+\cosh a +\cosh b+\cosh c)}\cosh OM=\cosh OA +\cosh OB +\cosh OC=3\cosh R$.
Approximating this we get the equation $3\left(1+\frac{R^2}{2}\right)= \sqrt{9+a^2+b^2+c^2}\left(1+\frac{OM^2}{2}\right)= 3\sqrt{1+\frac{a^2+b^2+c^2}{9}}\left(1+\frac{OM^2}{2}\right)$.
The functions on the right hand side we approximate of order two. If we multiply these polynomials and hold only those terms which order at most 2 we can deduce the equation $1+\frac{R^2}{2}=1+\frac{a^2+b^2+c^2}{2\cdot 9}+\frac{OM^2}{2}$,
and hence we deduced the Euclidean formula $OM^2=R^2-\frac{a^2+b^2+c^2}{9}$.
\end{rem}

\begin{cor} Applying (4.18) to a triangle with four ideal circumcenters, we get a formula which determines the common distance of three points of a hypercycle from the basic line of it. In fact, if $d$ means the searched distance, than $\frac{2\sinh\frac{a}{2}\sinh\frac{b}{2} \sinh\frac{c}{2}}{n}= \tanh R=\tanh \left(d+\varepsilon\frac{\pi}{2}i\right)= \frac{\sinh \left(d+\varepsilon\frac{\pi}{2}i\right)}{\cosh\left(d+\varepsilon\frac{\pi}{2}i\right)}=\frac{\varepsilon i\cosh d}{\varepsilon i\sinh d }=\coth d$,
and we get:
\begin{equation}
\tanh d=\frac{n}{2\sinh\frac{a}{2}\sinh\frac{b}{2}\sinh\frac{c}{2}}.
\end{equation}
For the Euclidean analogy of this equation we can use the first order Taylor polynomial of the hyperbolic function. Our formula yields
to the following $\frac{1}{R}=d=\frac{4T}{abc}$
implying a well-known connection among the sides, the circumradius and the area of a triangle.
\end{cor}

\subsubsection{On the center of the inscribed and escribed cycles}

We are aware of the fact that the bisectors of the interior angles of a hyperbolic triangle are concurrent at a point $I$, called the incenter,  which is equidistant from the sides of the triangle. The radius of the \emph{incircle} or \emph{inscribed circle}, whose center is at the incenter and touches the sides, shall be designated by $r$. Similarly the bisector of any interior angle and those of the exterior angles at the other vertices, are concurrent at a point outside the triangle; these three points are called \emph{excenters}, and the corresponding tangent cycles \emph{excycles} or \emph{escribed cycles}. The excenter lying on $AI$ is denoted by $I_A$, and the radius of the escribed cycle with center at $I_A$ is $r_A$. We denote by $X_A$, $X_B$, $X_C$ the points of the interior bisectors meets $BC$, $AC$, $AB$, respectively. Similarly $Y_A$, $Y_B$ and $Y_C$ denote the intersection points of the exterior bisectors at $A$, $B$ and $C$ with $BC$, $AC$ and $AB$, respectively.
\begin{figure}[htbp]
\centerline{\includegraphics[scale=0.7]{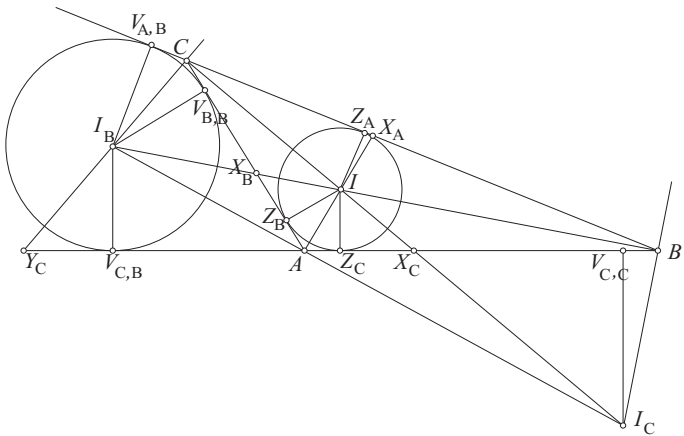}}
\caption{Incircles and excycles.}
\end{figure}
We note that the excenters and the points of intersection of the sides with the bisectors of the corresponding exterior angles could be points at infinity or could also be ideal points. Let $Z_A$, $Z_B$ and $Z_C$ denote the touching points of the incircle with the lines $BC$, $AC$ and $AB$, respectively and the touching points of the excycles with center $I_A$, $I_B$ and $I_C$ are given by the triples $\{V_{A,A},V_{B,A},V_{C,A}\}$, $\{V_{A,B},V_{B,B},V_{C,B}\}$ and $\{V_{A,C},V_{B,C},V_{C,C}\}$, respectively (see in Fig. 6).
\begin{thm}[\cite{gho 3}]
For the radii $r$, $r_A$, $r_B$ or $r_C$ we have the following formulas: .
\begin{equation} \tanh r=\frac{n}{\sinh s}, \quad \tanh r_A=\frac{n}{\sinh (s-a)},
\end{equation}
\begin{equation}
\tanh r=\frac{N}{2\cos \frac{\alpha}{2}\cos \frac{\beta}{2}\cos \frac{\gamma}{2}},
\end{equation}
\begin{eqnarray}
\coth r & = & \frac{\sin(\delta+\alpha)+\sin(\delta+\beta)+\sin(\delta+\gamma)+\sin\delta}{2N}, \\
\coth r_A & = & \frac{-\sin(\delta+\alpha)+\sin(\delta+\beta)+\sin(\delta+\gamma)-\sin\delta}{2N},
\end{eqnarray}
\begin{eqnarray}
\tanh R+\tanh R_A & = & \coth r_B+\coth r_C, \\
\nonumber
\tanh R_B+\tanh R_C & = & \coth r+\coth r_A, \\
\nonumber
\tanh R +\coth r & = & \frac{1}{2}\left(\tanh R+\tanh R_A+\tanh R_B+\tanh R_C\right),
\end{eqnarray}
\begin{eqnarray}
n_A(I):n_B(I):n_C(I) & = &  \sinh a:\sinh b:\sinh c,\\
n_A(I_A):n_B(I_A):n_C(I_A) & = &  -\sinh a :\sinh b: \sinh c.
\end{eqnarray}
\end{thm}
The following theorem describes relations between the distance of the incenter and circumcenter, the radii $r,R$ and the side-lengths $a,b,c$ .
\begin{thm}[\cite{gho 3}]
Let $O$ and $I$ be the center of the circumscribed and inscribed circles, respectively. Then we have
\begin{equation}
\cosh OI=2\cosh \frac{a}{2}\cosh \frac{b}{2}\cosh \frac{c}{2}\cosh r\cosh R+\cosh\frac{a+b+c}{2}\cosh(R-r).
\end{equation}
\end{thm}

\begin{rem}
The second order approximation of (4.28) leads to the equality
$1+\frac{OI^2}{2}=2\left(1+\frac{r^2}{2}\right)\left(1+\frac{R^2}{2}\right)\left(1+\frac{a^2}{8}\right)\left(1+\frac{b^2}{8}\right) \left(1+\frac{c^2}{8}\right)-
$
\hfill \break
$
-\left(1+\frac{(a+b+c)^2}{8}\right)\left(1+\frac{(R-r)^2}{2}\right)$.
From this we get that $OI^2=R^2+r^2+\frac{a^2+b^2+c^2}{4}-\frac{ab+bc+ca}{2}+2Rr$.
But for Euclidean triangles we have (see \cite{bell}) $a^2+b^2+c^2=2s^2-2(4R+r)r$ and  $ab+bc+ca=s^2+(4R+r)r$.
The equality above leads to the Euler's formula: $OI^2=R^2-2rR$.
\end{rem}

\end{document}